\documentclass[preprint,3p,times]{elsarticle}

\UseRawInputEncoding
\usepackage{amssymb}
\usepackage{amscd,amssymb,latexsym,srcltx}
\usepackage{dsfont}
\usepackage{amsthm}
\usepackage{setspace}
\usepackage{mdwlist}
\usepackage{enumerate}
\usepackage{graphics}
\usepackage{graphicx}
\usepackage{epsfig}
\usepackage{subfigure}
\usepackage{amssymb}
\usepackage{mathrsfs}
\usepackage{amsmath}
\allowdisplaybreaks
\usepackage{amsfonts}
\usepackage{geometry}
 \usepackage{hyperref}
\usepackage{geometry}
\usepackage{amsfonts}
\usepackage{mathrsfs}
\usepackage{amsmath}
\usepackage{amssymb}
\usepackage{amsthm}
\usepackage{geometry}
\newtheorem{assumption}{Assumption}[section]

\newtheorem{thm}{Theorem}[section]

\newtheorem{Def}[thm]{Definition}
\newtheorem{lemma}[thm]{Lemma}

\numberwithin{equation}{section}

\newcounter{cnum}
\makeatletter
\def\C{\@ifnextchar[{\@with}{\@without}}
\def\@with[#1]{\@ifundefined{c@#1}{%
    	   \newcounter{#1}%
            \@bsphack\protected@write\@auxout{}{
            \string\newlabel{#1}{{\thecnum} {\thepage}}
            \@esphack
            }
            C_{\thecnum}
            \stepcounter{cnum}%
        }{C_{\ref{#1}}}}
\def\@without{C_{\thecnum}
        \stepcounter{cnum}}
\makeatother
\journal{}

\begin{document}

\begin{frontmatter}


\title{Existence of a generalized polynomial attractor for the wave equation with nonlocal weak damping, anti-damping and critical nonlinearity}
%\tnotetext[t1]{The work is supported by National Natural Science Foundation of China (No.11731005; No.11801071).}
%\tnotetext[t1]{The work is supported by National Natural Science Foundation of China (No.11731005; No.11801071).}

\author[1,2]{Chunyan Zhao}

\author[2]{Chengkui Zhong\corref{cor1}}
\author[2]{Senlin Yan}

\address[1]{School of Mathematics and Big Data, Anhui University of Science and Technology, Huainan, 232001, PR China}
\address[2]{Department of Mathematics, Nanjing University, Nanjing, 210093, PR China}

 \cortext[cor1]{Corresponding author.\\
 E-mails: emmanuelz@163.com(C. Zhao), ckzhong@nju.edu.cn(C. Zhong), dg20210019@smail.nju.edu.cn(S. Yan)}

\begin{abstract}
In this paper, we first establish a criterion based on contractive function for the existence of generalized polynomial attractors. This criterion  only involves some rather weak compactness associated with the repeated inferior limit  and requires no compactness,  which makes it  suitable for critical cases. Then by this abstract theorem, we verify the existence of a generalized polynomial attractor and estimate its attractive speed for the wave equation with nonlocal weak damping, anti-damping and critical nonlinearity.
\end{abstract}

\begin{keyword}

Polynomial attractor; Wave equation; Nonlocal weak damping; Critical nonlinearity
\MSC[2010] 35B40\sep 35B41\sep 35L05\sep 35B33
\end{keyword}

\end{frontmatter}

% \linenumbers

\section{Introduction}
The global attractor is considered as an appropriate concept to describe the long-time behavior of a dissipative infinite dimensional dynamical system. Even so, it still has some limitations. There is no
quantitative information on the speed at which it attracts  bounded sets in its definition. In fact, it may attract  bounded sets at arbitrarily low rates. In order to quantitatively describe the attractive velocity of attractors, we introduced the concept of the generalized~$\varphi$-attractor in \cite{my4}. A generalized~$\varphi$-attractor is a positively invariant, compact attracting set whose attractive speed is characterized by a non-negative decay function~$\varphi$. In particular, it is a generalized polynomial attractor when~$\varphi$ is a polynomial function. By estimating the decay rate of noncompactness measure, we verified the existence of a generalized polynomial attractor and estimated its specific attractive rate  for a semilinear wave equation with nonlocal weakly damping, anti-damping and subcritical nonlinearity
\begin{equation*}
  u_{tt}-\Delta u+k||u_{t}||^p u_t+f(u)
=\displaystyle\int_{\Omega}K(x,y)u_{t}(t,y)dy+h(x) \ \ \text{in}\ [0,\infty)\times\Omega.
\end{equation*}
However, a key ingredient in the method in \cite{my4} is the compactness condition.
When~$f$ is of critical growth,  the corresponding Sobolev embedding is no longer compact, and  thus we failed to establish the existence of a  generalized polynomial attractor for the critical case in \cite{my4}. The aim of this paper is to solve this problem. Chueshov\cite{Chueshov2015,Chueshov2008} proposed the method of contractive functions by which the existence of the global attractor can be established for the critical case. We are inspired to  establish a criterion based on contractive functions for the existence of generalized polynomial attractors. This criterion  only involves some rather weak compactness associated with the repeated inferior limit  and requires no compactness,  which makes it  suitable for critical cases. Then we apply this abstract theorem to the above wave equation with critical nonlinearity~$f$ and solve our problem.

The paper is organized as follows. Section 2 presents an abstract theorem on the existence of generalized polynomial attractors. In Section 3, we apply this theorem to the above wave equation in the critical case.

 Throughout this paper, the symbols~$\hookrightarrow$ and~$\hookrightarrow\hookrightarrow$ stand for continuous embedding and compact embedding respectively. The capital letter ``C" with a (possibly empty) set of subscripts will denote a positive constant depending only on its subscripts and may vary from one occurrence to another.

 \section{An abstract theorem on the existence of polynomial attractors}
\begin{Def}\citep[Definition 2.6]{my4}\label{def21-1-31-1}
Assume that~$\varphi:\mathds{R}^{+}\rightarrow\mathds{R}^{+}$ satisfies~$\varphi(t)\rightarrow 0$~as~$t\rightarrow +\infty$. We call a compact set ~$\mathcal{A}^*\subseteq X$ a generalized~$\varphi$-attractor for the dynamical system~$(X, \{S(t)\}_{t\geq0})$, iff~$\mathcal{A}^*$ is positively invariant with respect to~$S(t)$ and there exists~$t_{0}\in \mathds{R}$ such that for every bounded set~$B\subseteq X$ there exists~$t_{B}\geq0$ such that
\begin{equation*}
\begin{split}
\mathrm{ dist}\left(S(t)B, \mathcal{A}^*\right)\leq \varphi(t+t_0-t_{B}),\ \forall t\geq t_{B}.
  \end{split}
\end{equation*}
In particular, if~$\varphi(t)=Ct^{-\beta}$ for certain positive constants~$C,\beta$, then~$\mathcal{A}^*$ is called a generalized polynomial attractor.
\end{Def}
\begin{lemma}\citep[Theorem 2.2]{my4}\label{20-8-5-3}
Assume that the dissipative dynamical system~$(X, \{S(t)\}_{t\geq0})$ is~$\varphi$-decaying with respect to the noncompactness measure~$\alpha$, i.e., there exist a positively invariant bounded absorbing set~$\mathcal{B}_0$ and a positive constant~$t_{0}$ such that
\begin{equation*}
  \alpha(S(t)\mathcal{B}_0)\leq\varphi(t),\ \forall t\geq t_{0}.
\end{equation*}
Then~$(X, \{S(t)\}_{t\geq0})$ possesses a generalized~$\varphi$-attractor~$\mathcal{A}^*$ such that for every bounded set~$B\subseteq X$ we have
\begin{equation}\label{20-8-5-33}
\begin{split}
\mathrm{ dist}\left(S(t)B, \mathcal{A}^*\right)\leq \varphi(t-t_{*}(B)-1),\ \forall t\geq t_{*}(B)+t_{0}+1,
  \end{split}
\end{equation}
where~$t_{*}(B)$ is the entering time of~$B$~into~$\mathcal{B}_0$.
\end{lemma}

\begin{Def}
Let~$X$ be a complete metric space and~$B$ be a bounded subset of~$X$. Function~$\Psi:X\times X\rightarrow \mathds{R^{+}}$ is said to be contractive on~$B\times B$ if for any $\{y_n\}\subseteq B$ there exists a subsequence $\{y_{n_k}\}$ of $\{y_n\}$ such that
\begin{equation*}
 \lim_{k\rightarrow\infty} \lim_{l\rightarrow\infty} \Psi(y_{n_k},y_{n_l})=0.
\end{equation*}
We denote by~$\mathrm{Contr}(B)$ the set of all contractive functions on~$B\times B$.
\end{Def}

\begin{thm}\label{20-7-27-80}
Let~$\{S(t)\}_{t\geq0}$ be a dissipative dynamical system on a complete metric space~$(X,d)$ and~$\mathcal{B}_0$ be a positively invariant bounded absorbing set. Assume that there exist $\beta\in(0,1)$, positive constants~$C, T$, functions~$g_l:(\mathds{R}^{+})^{m}\rightarrow\mathds{R}^{+}\ (l=1,2)$,~$\Psi_r:X\times X\rightarrow \mathds{R^{+}}\ (r=1,2)$ and pseudometrics ~$\varrho_{T}^{i}\ (i=1,2,\ldots,m)$ on~$\mathcal{B}_0$ such that
\begin{itemize}
\item[(i)]~$g_l$~is non-decreasing with respect to each
variable,~$g_l(0,\ldots,0)=0$~and~$g_l$~is continuous at~$(0,\ldots,0)$;
\item[(ii)]~$\varrho_{T}^{i}(i=1,2,\ldots,m)$ is precompact on~$\mathcal{B}_0$, i.e., any sequence~$\{x_n\}\subseteq \mathcal{B}_0$ has a subsequence~$\{x_{n_k}\}$~which is Cauchy with
respect to~$\varrho_{T}^{i}$;
\item[(iii)]~$\Psi_{r}\in \mathrm{Contr}(\mathcal{B}_0)\ (r=1,2)$;
\item[(iv)]~for all~$y_1,y_2\in \mathcal{B}_0$, we have the inequalities
\begin{equation}\label{20-8-3-20}
\begin{split}
(d(S(T)y_1,S(T)y_2))^2\leq
(d(y_1,y_2))^2+g_1\Big(\varrho_{T}^1\big(y_1,y_2\big),\varrho_{T}^2\big(y_1,y_2\big),\ldots,\varrho_{T}^m\big(y_1,y_2\big)\Big)+\Psi_{1}\big(y_1,y_2\big)
\end{split}
\end{equation}
and
\begin{equation}\label{20-7-27-25}
\begin{split}
\left(d(S(T)y_1,S(T)y_2)\right)^2
 \leq &C\bigg[ (d(y_1,y_2))^2-(d(S(T)y_1,S(T)y_2))^2+g_1\Big(\varrho_{T}^1\big(y_1,y_2\big),\varrho_{T}^2\big(y_1,y_2\big),\ldots,\varrho_{T}^m\big(y_1,y_2\big)\Big)\\&+\Psi_{1}\big(y_1,y_2\big)\bigg]^{\beta}
+g_2\Big(\varrho_{T}^1\big(y_1,y_2\big),\varrho_{T}^2\big(y_1,y_2\big),\ldots,\varrho_{T}^m\big(y_1,y_2\big)\Big)+\Psi_{2}\big(y_1,y_2\big).
\end{split}
\end{equation}

\end{itemize}
Then  there exists~$N_0\in\mathds{N}$ such that
\begin{equation}\label{21-12-23-35}
\begin{split}
\alpha(S(t)\mathcal{B}_0)\leq2\bigg[(\frac{t}{T}-N_0-1)(\frac{1}{\beta}-1)(1+3C)^{-\frac{1}{\beta}}+\big(\alpha(\mathcal{B}_0)\big)^{\frac{2(\beta-1)}{\beta}}\bigg]^{\frac{\beta}{2(\beta-1)}}, ~\forall t\geq (N_0+1)T.
\end{split}
\end{equation}

Thus,~$(X, \{S(t)\}_{t\geq0})$ possesses a generalized polynomial attractor~$\mathcal{A}^*$ and for each bounded subset $B$ of $X$ we have
\begin{equation}\label{21-9-4-33}
\begin{split}
\mathrm{ dist}\left(S(t)B, \mathcal{A}^*\right) \leq2\bigg[(\frac{t-t_{*}(B)-1}{T}-N_0-1)(\frac{1}{\beta}-1)(1+3C)^{-\frac{1}{\beta}}+\big(\alpha(\mathcal{B}_0)\big)^{\frac{2(\beta-1)}{\beta}}\bigg]^{\frac{\beta}{2(\beta-1)}}, ~\forall t\geq (N_0+1)T+t_{*}(B)+1,
\end{split}
\end{equation}
where~$t_{*}(B)$ is the entering time of~$B$~into~$\mathcal{B}_0$.
\end{thm}
\begin{proof}\textbf{Step1:} For any~$y_1,y_2\in \mathcal{B}_0$, we deduce from~$(\ref{20-8-3-20})$ and $(\ref{20-7-27-25})$ that

\begin{equation}\label{21-12-4-6}
\begin{split}
&q\Big(\big(d(S(T)y_1,S(T)y_2)\big)^2\Big)\equiv(3C)^{-\frac{1}{\beta}}\big(d(S(T)y_1,S(T)y_2)\big)^{\frac{2}{\beta}}+\big(d(S(T)y_1,S(T)y_2)\big)^{2}\\\leq &\big(d(y_1,y_2)\big)^2+g_1\Big(\varrho_{T}^1\big(y_1,y_2\big),\ldots,\varrho_{T}^m\big(y_1,y_2\big)\Big)+C^{-\frac{1}{\beta}}g_2^{\frac{1}{\beta}}\Big(\varrho_{T}^1\big(y_1,y_2\big),\ldots,\varrho_{T}^m\big(y_1,y_2\big)\Big) +\Psi_1\big(y_1,y_2\big)+C^{-\frac{1}{\beta}}\Psi_2^{\frac{1}{\beta}}\big(y_1,y_2\big)
\end{split}
\end{equation}
with~$q(s)=(3C)^{-\frac{1}{\beta}}s^{\frac{1}{\beta}}+s,\ s\geq 0$. We denote by~$w$
the inverse function of~$q$ on~$\mathds{R}^+$.

Since~$w$ is increasing,~$(\ref{21-12-4-6})$ implies that
\begin{equation}\label{21-9-4-10}
\begin{split}
\big(d(S(T)y_1,S(T)y_2)\big)^2\leq w\bigg(\big(d(y_1,y_2)\big)^2+g\Big(\varrho_{T}^1\big(y_1,y_2\big),\ldots,\varrho_{T}^m\big(y_1,y_2\big)\Big) +\Psi\big(y_1,y_2\big)
\bigg),
\end{split}
\end{equation}
where~$g\Big(\varrho_{T}^1\big(y_1,y_2\big),\ldots,\varrho_{T}^m\big(y_1,y_2\big)\Big)\equiv g_1\Big(\varrho_{T}^1\big(y_1,y_2\big),\ldots,\varrho_{T}^m\big(y_1,y_2\big)\Big)+C^{-\frac{1}{\beta}}g_2^{\frac{1}{\beta}}\Big(\varrho_{T}^1\big(y_1,y_2\big),\ldots,\varrho_{T}^m\big(y_1,y_2\big)\Big) $ and $\Psi\big(y_1,y_2\big)\equiv \Psi_{1}\big(y_1,y_2\big)+C^{-\frac{1}{\beta}}\Psi_{2}^{\frac{1}{\beta}}\big(y_1,y_2\big)$.

Iterating~$(\ref{21-9-4-10})$ we have that
\begin{equation}\label{21-12-22-7}
\begin{split}
\big(d(S(nT)y_1,S(nT)y_2)\big)^2\leq & w\bigg(\big(d(S((n-1)T)y_1,S((n-1)T)y_2)\big)^2+\Psi\big(S((n-1)T)y_1,S((n-1)T)y_2\big)\\&\\&+g\Big(\varrho_{T}^1\big(S((n-1)T)y_1,S((n-1)T)y_2\big),\ldots,\varrho_{T}^m\big(S((n-1)T)y_1,S((n-1)T)y_2\big)\Big)\bigg)\\&
\\\leq&w\bigg( w\bigg(\cdots w\bigg(w\bigg(\big(d(y_1,y_2)\big)^2+g\Big(\varrho_{T}\big(y_1,y_2\big),\ldots,\varrho_{T}\big(y_1,y_2\big)\Big)  +\Psi\big(y_1,y_2\big) \bigg)\\&+\Psi\big(S(T)y_1,S(T)y_2\big)+g\Big(\varrho_{T}^1\big(S(T)y_1,S(T)y_2\big),\ldots,\varrho_{T}^m\big(S(T)y_1,S(T)y_2\big)\Big)\bigg)+\cdots\bigg)\\&+g\Big(\varrho_{T}^1\big(S((n-1)T)y_1,S((n-1)T)y_2\big),\ldots,\varrho_{T}^m\big(S((n-1)T)y_1,S((n-1)T)y_2\big)\Big)\\&+\Psi\big(S((n-1)T)y_1,S((n-1)T)y_2\big)\bigg),~~\forall n\in\mathds{N}.
\end{split}
\end{equation}

The right-hand side of the relation above is a continuous function of $d(y_1,y_2)$ and
\begin{equation*}
 \Psi\big(S(kT)y_1,S(kT)y_2\big),~g\Big(\varrho_{T}^1\big(S(kT)y_1,S(kT)y_2\big),\ldots,\varrho_{T}^m\big(S(kT)y_1,S(kT)y_2\big)\Big)~~(k=0,\ldots,n-1).
\end{equation*}

For each ~$B\subseteq \mathcal{B}_0$ and any $\epsilon>0$, by the definition of the noncompactness measure~$\alpha$, there exist sets~$F_1,F_2,\ldots,F_\gamma$~such that
\begin{equation}\label{21-4-19-1}
  B\subseteq\cup_{j=1}^{\gamma}F_j,\ \mathrm{diam} F_j<\alpha(B)+\epsilon.
\end{equation}

For every sequence $\{y_p\}\subseteq B$, there exist~$j_{*}\in\{1,2,\ldots,\gamma\}$ and a subsequence~$\{y_{p_\iota}\}$ of $\{y_p\}$ such that
$\{y_{p_\iota}\}_{\iota=1}^{+\infty}\subseteq F_{j_*}$.

By~$(\ref{21-4-19-1})$, we have
\begin{equation}\label{21-4-19-11}
  d\big(y_{p_\iota},y_{p_\nu}\big)\leq \mathrm{diam} F_{j_*}<\alpha(B)+\epsilon, ~~\forall \iota,\nu\in \mathds{N}.
\end{equation}

Since~$\varrho_{T}^{i}(i=1,2,\ldots,m)$ is precompact on~$\mathcal{B}_0$ and~$\Psi_{r}\in \mathrm{Contr}(\mathcal{B}_0)\ (r=1,2)$, ~$\{y_{p_\iota}\}$ has a subsequence ~$\{y_{p_{\iota_{\lambda}}}\}$  such that
\begin{equation}\label{21-12-22-13}
 \lim_{\lambda,\chi\rightarrow+\infty} g\Big(\varrho_{T}^1\big(S(kT)y_{p_{\iota_{\lambda}}},S(kT)y_{p_{\iota_{\chi}}}\big),\ldots,\varrho_{T}^m\big(S(kT)y_{p_{\iota_{\lambda}}},S(kT)y_{p_{\iota_{\chi}}}\big)\Big)\bigg)=0
\end{equation}
and
\begin{equation}\label{21-12-23-33}
 \lim_{\lambda\rightarrow+\infty} \lim_{\chi\rightarrow+\infty}  \Psi\big(S(kT)y_{p_{\iota_{\lambda}}},S(kT)y_{p_{\iota_{\chi}}}\big)=0
\end{equation}
hold for~$k=0,\ldots,n-1$.

We deduce from~$(\ref{21-12-22-7})$,~$(\ref{21-4-19-11})$-$(\ref{21-12-23-33})$ that
\begin{equation}\label{21-12-22-8}
    \liminf_{p\rightarrow +\infty} \liminf_{q\rightarrow +\infty}  \big(d(S(nT)y_{p},S(nT)y_{q})\big)^2\leq\liminf_{\lambda\rightarrow+\infty} \liminf_{\chi\rightarrow+\infty}  \big(d(S(nT)y_{p_{\iota_{\lambda}}},S(nT)y_{p_{\iota_{\chi}}})\big)^2\leq  w^n\Big(\big(\alpha(B)+\epsilon\big)^2\Big),
\end{equation}
where $w^n$ denotes the n-fold composition $w\circ w \circ \cdots \circ w$.
By the arbitrariness of~$\epsilon$,~$(\ref{21-12-22-8})$ implies that
\begin{equation}\label{21-12-22-16}
    \liminf_{p\rightarrow +\infty} \liminf_{q\rightarrow +\infty}  \big(d(S(nT)y_{p},S(nT)y_{q})\big)^2\leq w^n\Big(\big(\alpha(B)\big)^2\Big),~~\forall\{y_p\}\subseteq B.
\end{equation}
% It is easy to check that~$w(0)=0$ and~$w(s)<s,\ s>0$, and thus $w^n\Big(\big(\alpha(B)\big)^2\Big)\rightarrow 0$ as $n\rightarrow+\infty$. Consequently, according to Proposition 2.2.18 in \cite{Chueshov2015},~$(\ref{21-12-22-16})$ implies that
% \begin{equation}\label{21-9-4-15}
%  \alpha(S(t)B)\rightarrow 0\ \text{as}\ t\rightarrow+\infty,
%\end{equation}
%i.e.,~$\{S(t)\}_{t\geq0}$~is~$\omega$-limit compact.

\textbf{Step2:} Inequality~$(\ref{21-12-22-16})$ implies that
\begin{equation}\label{21-12-23-17}
\begin{split}
\Big(\alpha\big(S(nT)B\big)\Big)^2\leq 4w^n\Big(\big(\alpha(B)\big)^2\Big).
 \end{split}\end{equation}
 If~$(\ref{21-12-23-17})$ was false,  then there would exist~$d_0>0$ such that
\begin{equation}\label{21-9-1-18}
\begin{split}\Big(\alpha\big(S(nT)B\big)\Big)^2>d_0>  4w^n\Big(\big(\alpha(B)\big)^2\Big).
 \end{split}\end{equation}

Hence according to the definition of the noncompactness measure~$\alpha$,~$S(nT)B$  has no finite cover of diameter less than or equal to $\sqrt{d_0}$. Consequently, for any $y_1\in B$, there exists~$y_2\in B$~such that
\begin{equation*}
  d(S(nT)y_1,S(nT)y_2)>\frac{\sqrt{d_0}}{2}.
\end{equation*}
And then we can find ~$y_3\in B$~ satisfying
\begin{equation*}
  d(S(nT)y_3,S(nT)y_i)>\frac{\sqrt{d_0}}{2},\ i=1,2.
\end{equation*}
 Iterating in this way, we obtain a sequence $\{y_p\}\subseteq B$ satisfying
\begin{equation}\label{21-9-1-16}
  d(S(nT)y_p,S(nT)y_q)>\frac{\sqrt{d_0}}{2}>\sqrt{w^n\Big(\big(\alpha(B)\big)^2\Big)},\  \forall p\neq q.
\end{equation}
This is contrary to~$(\ref{21-12-22-16})$.

\textbf{Step3:}
Let $y(0)\in \mathds{R}^+$ and $y(n)=w^n\big(y(0)\big),~\forall n\in \mathds{N}$, which is  equivalent to
\begin{equation}\label{21-12-25-3}
  y(n)=w\big(y(n-1)\big),\forall n\in\mathds{N}.
\end{equation}
It is easy to check that~$w(0)=0$ and~$w(s)<s,\ s>0$. Hence $y(n)$ is  monotonically decreasing and $y(n)\rightarrow 0$ as $n\rightarrow+\infty$. Consequently, there exists $N_0\in \mathds{N}$ such that $y(n-1)-y(n)\in (0,1),~\forall n\geq N_0$.
By the definition of $w$, $(\ref{21-12-25-3})$ can be rewritten as
\begin{equation*}
  y(n)=3C[y(n-1)-y(n)]^{\beta}, ~\forall n\in \mathds{N},
\end{equation*}
which implies that
\begin{equation}\label{21-12-25-9}
  y(n-1)=y(n)+y(n-1)-y(n)<(1+3C)[y(n-1)-y(n)]^{\beta},~\forall n\geq N_0.
\end{equation}
We deduce from~$(\ref{21-12-25-9})$ that
\begin{equation}\label{21-12-25-11}
\begin{split}
  y^{1-\frac{1}{\beta}}(n)-y^{1-\frac{1}{\beta}}(n-1)=&(1-\frac{1}{\beta})\big(\theta y(n)+(1-\theta)y(n-1)\big)^{-\frac{1}{\beta}}\big(y(n)-y(n-1)\big)\\ \geq&(1-\frac{1}{\beta})y^{-\frac{1}{\beta}}(n-1)\big(y(n)-y(n-1)\big)\\
  \geq& \big(\frac{1}{\beta}-1\big)(1+3C)^{-\frac{1}{\beta}},~\forall n\geq N_0,
  \end{split}
\end{equation}
where $\theta\in(0,1)$.
It follows from~$(\ref{21-12-25-11})$ that
\begin{equation}\label{21-12-25-18}
y(n)=w^n\big(y(0)\big)\leq \bigg[(n-N_0)(\frac{1}{\beta}-1)(1+3C)^{-\frac{1}{\beta}}+y^{1-\frac{1}{\beta}}(N_0)\bigg]^{\frac{\beta}{\beta-1}}, ~\forall n\geq N_0.
\end{equation}

We deduce from $(\ref{21-12-23-17})$ and $(\ref{21-12-25-18})$ that
\begin{equation}\label{21-12-23-28}
\begin{split}
\alpha(S(nT)\mathcal{B}_0)\leq 2\bigg[(n-N_0)(\frac{1}{\beta}-1)(1+3C)^{-\frac{1}{\beta}}+\big(\alpha(\mathcal{B}_0)\big)^{\frac{2(\beta-1)}{\beta}}\bigg]^{\frac{\beta}{2(\beta-1)}}, ~\forall n\geq N_0.
\end{split}
\end{equation}
It follows from $(\ref{21-12-23-28})$ that
\begin{equation}\label{21-12-23-31}
\begin{split}
\alpha(S(t)\mathcal{B}_0) \leq\alpha(S([\frac{t}{T}]T)\mathcal{B}_0)\leq2\bigg[(\frac{t}{T}-N_0-1)(\frac{1}{\beta}-1)(1+3C)^{-\frac{1}{\beta}}+\big(\alpha(\mathcal{B}_0)\big)^{\frac{2(\beta-1)}{\beta}}\bigg]^{\frac{\beta}{2(\beta-1)}}, ~\forall t\geq (N_0+1)T.
\end{split}
\end{equation}
where $[\frac{t}{T}]$ denotes the integer part of $\frac{t}{T}$.

By Lemma~$\ref{20-8-5-3}$,~$(\ref{21-12-23-31})$ implies the existence of a polynomial attractor~$\mathcal{A}^*$ which satisfies~$(\ref{21-9-4-33})$.

\end{proof}

\section{Existence of a polynomial attractor and estimate of its attractive speed for the wave equation in the critical case}
Let~$\Omega$ be a bounded domain in~$\mathbb{R}^N(N\geq3)$ with a sufficiently smooth boundary~$\partial\Omega$. We denote the inner product and the norm on~$L^2(\Omega)$ by~$(\cdot,\cdot)$ and~$\|\cdot\|$ respectively, and the norm on~$L^p(\Omega)$ by~$\|\cdot\|_p$. Consider the initial-boundary value problem of a wave equation:
\begin{align}
\label{wave equa}&u_{tt}-\Delta u+k||u_{t}||^p u_t+f(u)
=\displaystyle\int_{\Omega}K(x,y)u_{t}(t,y)dy+h(x) \ \ \text{in}\ [0,\infty)\times\Omega,\\
\label{condition}&u|_{\partial\Omega}=0,\ \ u(x,0)=u_0(x),\ \  u_{t}(x,0)=u_{1}(x),\  x\in\Omega.
\end{align}
\begin{assumption}\label{21-8-29-8}
\begin{itemize}
  \item [(i)]~$k$ and~$p$ are positive constants,~$K\in L^2(\Omega\times\Omega)$,~$h\in L^2(\Omega)$;
  \item [(ii)]~$f\in C^1(\mathbb{R})$~satisfies
  \begin{align}
\label{growth}&|f'(s)| \leq M(|s|^{\frac{2}{N-2}}+1),\\
\label{dissipativity condition}&\liminf_{|s|\rightarrow\infty}f'(s)\equiv \mu > -\lambda_{1},
\end{align}
where~$M\geq0$ and~$\lambda_{1}$ is the first eigenvalue of the operator~$-\Delta$ equipped with Dirichlet boundary condition.
\end{itemize}
\end{assumption}
\begin{lemma}\cite{my1}\label{20-9-30-1}
Let~$T> 0$ be arbitrary. Under Assumption~$\ref{21-8-29-8}$, problem~$(\ref{wave equa})$-$(\ref{condition})$ has a unique weak solution~$u \in C([0,T];H^1_0(\Omega)) \cap C^1([0,T];L^2(\Omega))$ for every~$(u_0,u_1) \in H^1_0(\Omega)\times L^2(\Omega)$, which generates the semigroup~$S(t)(u_0,u_1) = (u(t),u_t (t))\ (t\geq0)$ on~$H^1_0(\Omega)\times L^2(\Omega) $.
Furthermore, the semigroup~$\{S(t)\}_{t\geq0}$ is dissipative, which implies the existence of a positively invariant bounded absorbing set~$\mathcal{B}_0$.
\end{lemma}
%\begin{lemma}
%Let $\left(H,(\cdot,\cdot)_H\right)$ be an inner product space with the induced norm $\|\cdot\|_H$. Then there exists some positive constant $C_{\gamma}$ such that
%\begin{eqnarray}\label{2.11}
%&\Big(\|u\|_{H}^{\gamma-2}u-\|v\|_{H}^{\gamma-2}v,u-v\Big)_{H}
%&\geq
% \begin{cases}
% C_{\gamma}\|u-v\|^{\gamma}_{H}, ~ if\  \gamma\geq2,\\
%C_{\gamma}\frac{\|u-v\|_{H}^{2}}{(\|u\|_{H}+\|v\|_{H})^{2-\gamma}},~ if\  1<\gamma<2,
% \end{cases}\  \bigg(\forall u,v\in H, (u,v)\neq(0,0)\bigg).
%\end{eqnarray}
%\end{lemma}

\begin{thm}\label{21-9-4-2}
Under Assumption~$\ref{21-8-29-8}$, the dynamical system~$(H_0^1(\Omega)\times L^2(\Omega),\{S(t)\}_{t\geq0})$\ generated by problem~$(\ref{wave equa})$-$(\ref{condition})$ possesses a generalized polynomial attractor~$\mathcal{A}^*$ such that for every bounded set~$B\subseteq X$ we have
\begin{equation}\label{21-4-24-1}
\begin{split}
\mathrm{ dist}\left(S(t)B, \mathcal{A}^*\right)\leq  2\Big\{(\alpha(\mathcal{B}_0))^{-p}+pkC_p6^{-\frac{p+2}{2}}(t-t_{*}(B)-1)\Big\}^{-\frac{1}{p}},\ \forall t> t_{*}(B)+1,
\end{split}
\end{equation}
where~$t_{*}(B)$ is the entering time of~$B$ into~$\mathcal{B}_0$.
\end{thm}

\begin{proof}
Let~$\mathcal{B}_0\subseteq H_{0}^{1}(\Omega)\times L^{2}(\Omega)$ be a positively invariant bounded absorbing set. Write~$\Psi(u_t(t,x))=\int_{\Omega}K(x,y)u_t(t,y)dy$.
Let~$w(t),v(t)$ be two weak solutions to~$(\ref{wave equa})$-$(\ref{condition})$ corresponding to initial
data~$y_1,y_2\in \mathcal{B}_0$:~$(w(t),w_t(t))\equiv S(t)y_1, (v(t),v_t(t))\equiv S(t)y_2, \ y_1,y_2\in \mathcal{B}_0$. Since~$\mathcal{B}_0$ is positively invariant, we have
\begin{equation}\label{21-9-2-52}
\|(w(t),w_t(t))\|_{ H^1_0(\Omega)\times L^2(\Omega)}\leq C,~\|(v(t),v_t(t))\|_{ H^1_0(\Omega)\times L^2(\Omega)}\leq C,\forall t>0, y_1,y_2\in \mathcal{B}_0.
\end{equation}
The difference~$z(t)=w(t)-v(t)$ satisfies
\begin{equation}\label{21-9-2-53}
z_{tt}-\Delta z +k(||w_t||^p w_t-||v_t||^p v_t)+f(w)-f(v)=\Psi(z_t).
 \end{equation}
Let~$E_z(t)=\frac{1}{2}(||z_t(t)||^2+||\nabla z(t)||^2)$ and~$T$ be an arbitrary positive constant. Multiplying~$(\ref{21-9-2-53})$ by~$z_t$~in~$L^2(\Omega)$ and integrating from~$t$ to~$T$, we obtain
 \begin{equation}\label{21-9-2-55}
\begin{split}
E_z(t)=E_z(T)+k\int_t^T(||w_t||^p w_t-||v_t||^p v_t,z_t)d\tau+\int_t^T(f(w)-f(v),z_t)d\tau-\int_t^T(\Psi(z_t),z_t)d\tau.
\end{split}
 \end{equation}

Integrating~$(\ref{21-9-2-55})$ from~$0$ to~$T$ yields
 \begin{equation}\label{21-9-2-56}
\begin{split}
TE_z(T)=\int_0^TE_z(t)dt-\int_0^T\int_t^Tk(||w_t||^p w_t-||v_t||^p v_t,z_t)d\tau dt-\int_0^T\int_t^T(f(w)-f(v),z_t)d\tau dt +\int_0^T\int_t^T(\Psi(z_t),z_t)d\tau dt.
\end{split}
 \end{equation}
Multiplying~$(\ref{21-9-2-53})$ by~$z$ in~$L^2(\Omega)$ and integrating from~$0$ to~$T$ yields
 \begin{equation}\label{21-9-2-58}
\begin{split}
\int_0^T E_z(t)dt=-\frac{1}{2}(z_t,z)|^T_0+ \int_0^T ||z_t||^2 dt- \frac{1}{2} \int_0^T(f(w)-f(v),z) dt +  \frac{1}{2} \int_0^T(\Psi(z_t),z)  dt- \frac{k}{2} \int_0^T(||w_t||^p w_t-||v_t||^p v_t,z)dt.
\end{split}
\end{equation}
Substituting~$(\ref{21-9-2-58})$ into~$(\ref{21-9-2-56})$, we have
 \begin{equation}\label{20-7-25-7}
\begin{split}
TE_z(T)=&-\frac{1}{2}(z_t,z)|^T_0+ \int_0^T ||z_t||^2 dt- \frac{1}{2} \int_0^T(f(w)-f(v),z) dt +  \frac{1}{2} \int_0^T(\Psi(z_t),z)  dt- \frac{k}{2} \int_0^T(||w_t||^p w_t-||v_t||^p v_t,z)dt\\&-\int_0^T\int_t^Tk(||w_t||^p w_t-||v_t||^p v_t,z_t)d\tau dt-\int_0^T\int_t^T(f(w)-f(v),z_t)d\tau dt +\int_0^T\int_t^T(\Psi(z_t),z_t)d\tau dt.
\end{split}
 \end{equation}
 By Lemma 4.2 in \cite{my4}, we have
 \begin{equation*}
(||w_t||^p w_t-||v_t||^p v_t,w_t-z_t)\geq C_p||w_t-v_t||^{p+2},
 \end{equation*}
which, together with the concavity of~$g(s)=s^{\frac{2}{p+2}}~(s>0)$, yields that
   \begin{equation}\label{20-7-25-14}
   \begin{split}
    \int_0^T ||z_t||^2 dt\leq &C_p^{-\frac{2}{p+2}}\int_0^T (||w_t||^p w_t-||v_t||^p v_t,w_t-v_t)^{\frac{2}{p+2}}dt
    \leq C_p^{-\frac{2}{p+2}}T^{\frac{p}{p+2}}\Big(\int_0^T(||w_t||^p w_t-||v_t||^p v_t,w_t-v_t)dt\Big)^{\frac{2}{p+2}}.
 \end{split}\end{equation}
Taking~$t=0$ in~$(\ref{21-9-2-55})$ yields
 \begin{equation}\label{20-7-25-15}
\begin{split}
&\int_0^T(||w_t||^p w_t-||v_t||^p v_t,w_t-v_t)dt=\frac{1}{k}\left(E_z(0)-E_z(T)-\int_0^T(f(w)-f(v),z_t)d\tau+\int_0^T(\Psi(z_t),z_t)d\tau\right).
\end{split}\end{equation}
We infer from~$(\ref{20-7-25-14})$ and~$(\ref{20-7-25-15})$ that
 \begin{equation}\label{20-7-25-16}
   \begin{split}
    \int_0^T ||z_t||^2 dt\leq (kC_p)^{-\frac{2}{p+2}}T^{\frac{p}{p+2}}\bigg(E_z(0)-E_z(T)-\int_0^T(f(w)-f(v),z_t)d\tau+\int_0^T(\Psi(z_t),z_t)d\tau\bigg)^{\frac{2}{p+2}}.
 \end{split}\end{equation}
 It is easy to obtain the following estimate: $-\frac{1}{2} \int_0^T(f(w)-f(v),z) dt \leq  TC\sup_{t\in[0,T]}\|z(t)\|$, $-\frac{1}{2}(z_t,z)|^T_0\leq C\sup_{t\in[0,T]}\|z(t)\|$, $\int_0^T(\Psi(z_t),z)  dt\leq TC\sup_{t\in[0,T]}\|z(t)\|$, $-\frac{k}{2} \int_0^T(||w_t||^p w_t-||v_t||^p v_t,z)dt\leq TC\sup_{t\in[0,T]}\|z(t)\|$ and $\int_0^T\int_t^T(\Psi(z_t),z_t)d\tau dt\leq TC\int_0^T\|\Psi(z_t(t))\|dt$.
 %\begin{equation}\label{1-10-5}
%\begin{split}
% &\|f(w)-f(v)\|=\left\{\int_{\Omega}\bigg[\int_0^1 f'(v+\theta(w-v))(w-v)d\theta\bigg]^2dx\right\}^{\frac{1}{2}},
%  \leq C\big(\|\nabla w\|^{\frac{2}{N-2}}+\|\nabla v\|^{\frac{2}{N-2}}+1\big)\|\nabla(w-v)\|
%  \leq C\\
% & - \frac{1}{2} \int_0^T(f(w)-f(v),z) dt \leq  TC\sup_{t\in[0,T]}\|z(t)\|,\\
% &-\frac{1}{2}(z_t,z)|^T_0
%\leq C\sup_{t\in[0,T]}\|z(t)\|,\\
%&\int_0^T(\Psi(z_t),z)  dt\leq TC\sup_{t\in[0,T]}\|z(t)\|,\\
%&- \frac{k}{2} \int_0^T(||w_t||^p w_t-||v_t||^p v_t,z)dt\leq TC\sup_{t\in[0,T]}\|z(t)\|,\\
%&-\int_0^T\int_t^Tk(||w_t||^p w_t-||v_t||^p v_t,z_t)d\tau dt\leq0,\\
%&\int_0^T\int_t^T(\Psi(z_t),z_t)d\tau dt\leq TC\int_0^T\|\Psi(z_t(t))\|dt.
%\end{split}\end{equation}
Plugging these inequalities and~$(\ref{20-7-25-16})$ into~$(\ref{20-7-25-7})$, we obtain
 \begin{equation}\label{21-9-2-66}
\begin{split}
TE_z(T)\leq &C(T+1)\sup_{t\in[0,T]}\|z(t)\|+TC\int_0^T\|\Psi(z_t(t))\|dt+\left|\int_0^T\int_t^T(f(w)-f(v),z_t)d\tau dt\right|\\&+(kC_p)^{-\frac{2}{p+2}}T^{\frac{p}{p+2}}\bigg(E_z(0)-E_z(T)+\left|\int_0^T(f(w)-f(v),z_t)d\tau\right|+C\int_0^T\|\Psi(z_t(t))\|dt\bigg)^{\frac{2}{p+2}}.
\end{split}
 \end{equation}
Let~$\mathcal{A}$ denote the strictly positive operator on~$L^2(\Omega)$ defined by~$\mathcal{A}=-\triangle$ with domain~$D(\mathcal{A})=H^2(\Omega)\cap H^1_0(\Omega)$.
Let~$V$ be the completion of~$L^2(\Omega)$ with respect to the norm~$\|\cdot\|_{V}$ given by~$\|\cdot\|_V=\|\Psi(\cdot)\|+\|\mathcal{A}^{-\frac{1}{2}}\cdot\|$ and~$W$ be the completion of~$L^2(\Omega)$ with respect to the norm~$\|\cdot\|_{W}$ given by~$\|\cdot\|_W=\|\mathcal{A}^{-\frac{1}{2}}\cdot\|$.
We infer from~$(\ref{wave equa})$,~$  \|f(w(t))\|\leq C$ and~$\|\Psi(w_{t}(t))\|\leq \|K\|_{L^2(\Omega\times\Omega)}\|w_{t}(t)\|\leq C$ that
\begin{equation*}
\begin{split}
  \|\mathcal{A}^{-\frac{1}{2}}w_{tt}(t)\|\leq&\|\nabla w(t)\|+k\|w_{t}(t)\|^p\|\mathcal{A}^{-\frac{1}{2}}w_{t}(t)\|
  +\|\mathcal{A}^{-\frac{1}{2}}\big(\Psi(w_{t}(t))+h-f(w(t))\big)\| \leq C.
  \end{split}
\end{equation*}
Hence,~$ \int_0^T\|\mathcal{A}^{-\frac{1}{2}}w_{tt}(t)\|dt\leq C_{T}$.
Besides, we have~$\int_0^T\|w_{t}(t)\|dt\leq C_{T}$ and~$L^2(\Omega)\hookrightarrow\hookrightarrow V\hookrightarrow W$.
By Corollary 4 in \cite{Simon1986},~$\varrho^1_{T}(y_1,y_2)=\int_0^T\|\Psi(z_t(t))\|dt$ is precompact on~$\mathcal{B}_0$. By~Arzel\`{a}-Ascoli~Theorem,~$ C\big([0,T],H_0^1(\Omega)\big)\cap C^1\big([0,T],L^2(\Omega)\big)\hookrightarrow\hookrightarrow C\big([0,T],L^2(\Omega)\big)$. Thus~$\rho^2_{T}(y_1,y_2)=\sup_{t\in[0,T]}\|z(t)\|$ is precompact on~$\mathcal{B}_0$.

Let~$\big(u^{(n)}(t),u^{(n)}_t(t)\big)=S(t)y^{(n)}~(\{y^{(n)}\}\subseteq \mathcal{B}_0)$,~$F(\mu)=\displaystyle\int_0^{\mu}f(\tau)d\tau$ and~$s\in (0,1)$. By Alaoglu's theorem and Corollary 4 in \cite{Simon1986}, we deduce from~$H^1_0(\Omega)\hookrightarrow\hookrightarrow H^s(\Omega)\hookrightarrow L^2(\Omega)$ that there exists a subsequence of~$\big\{(u^{(n)},u^{(n)}_t)\big\}_{n=1}^{\infty}$( still denoted by~$\big\{(u^{(n)},u^{(n)}_t)\big\}_{n=1}^{\infty}$) such that
\begin{eqnarray}\label{1-13-1}
 \begin{cases}
 (u^{(n)},u^{(n)}_t)\overset{\ast}\rightharpoonup (u,u_t)\ \  \text{in} \ L^{\infty}(0,T;H^1_0(\Omega)\times L^2(\Omega)),\\
 u^{(n)}\rightarrow u\ \   \text{in }\ C([0,T];H^s(\Omega)),
 \end{cases}\  \ \text{as}\ n\rightarrow\infty.
\end{eqnarray}
 By~$(\ref{growth})$ and~$(\ref{21-9-2-52})$, for all~$t\geq0$ we have
\begin{equation}\label{1-13-10}
\begin{split}
 \left|\int_{\Omega}F(u^{(n)}(t))dx-\int_{\Omega}F(u(t))dx\right|
\leq C\|u^{(n)}(t)-u(t)\|_{H^{s}(\Omega)}.
\end{split}
\end{equation}
Combining~$(\ref{1-13-1})$ and~$(\ref{1-13-10})$ gives
\begin{equation}\label{1-13-11}
  \int_{\Omega}F(u^{(n)}(t))dx\rightrightarrows\int_{\Omega}F(u(t))dx\ \text{as}\ n\rightarrow\infty.
\end{equation}
It follows from~$H^{N}(\Omega)\hookrightarrow L^{\infty}(\Omega)$ that~$L^1(\Omega)\hookrightarrow (L^{\infty}(\Omega))^{*}\hookrightarrow H^{-N}(\Omega)$. Hence we deduce from~$(\ref{growth})$ that
\begin{equation}\label{1-13-4}
\begin{split}
  \|\mathcal{A}^{-\frac{N}{2}}f(u^{(n)}(t))-\mathcal{A}^{-\frac{N}{2}}f(u(t))\|
\leq C\|f(u^{(n)}(t))-f(u(t))\|_{1}
\leq C\|u^{(n)}(t)-u(t)\|_{H^{s}(\Omega)},~ \forall t\geq0.
\end{split}
\end{equation}
Combining~$(\ref{1-13-1})$ and~$(\ref{1-13-4})$ gives~$\sup_{t\in[0,T]}\|\mathcal{A}^{-\frac{N}{2}}\big(f(u^{(n)}(t))-f(u(t))\big)\|\longrightarrow0 \ \text{as}\ n\rightarrow\infty$.
Consequently, for each fixed~$t\in[0,T]$ and each~$\varphi\in L^1\big(0,T;H^N(\Omega)\cap H^1_0(\Omega)\big)$, we have
~$\int_t^T\big(f(u^{(n)}(\tau))-f(u(\tau)),\varphi\big)d\tau\longrightarrow 0\ \text{as}\ n\rightarrow\infty$, which implies
   \begin{equation}\label{1-13-7}
f(u^{(n)})\overset{\ast}\rightharpoonup f(u)\ \text{in}\ L^{\infty}\big(t,T;L^2(\Omega)\big)\ \text{as}\ n\rightarrow\infty.
\end{equation}
From~$(\ref{1-13-1})$ and~$(\ref{1-13-7})$, we obtain
\begin{equation}\label{1-13-8}
\begin{split}
  \lim_{n\rightarrow\infty}\lim_{m\rightarrow\infty}\int_t^T\int_{\Omega}
  f(u^{(n)}(\tau))u^{(m)}_{t}(\tau)dxd\tau
  =\int_{\Omega}F(u(T))dx-\int_{\Omega}F(u(t))dx,\\
   \lim_{n\rightarrow\infty}\lim_{m\rightarrow\infty}\int_t^T\int_{\Omega}
  f(u^{(m)}(\tau))u^{(n)}_{t}(\tau)dxd\tau =\int_{\Omega}F(u(T))dx-\int_{\Omega}F(u(t))dx.
\end{split}\end{equation}

By Lebesgue's dominated convergence theorem, we deduce from~$(\ref{1-13-11})$ and~$(\ref{1-13-8})$ that
\begin{equation}\label{1-13-12}
\begin{split}
  &\lim_{n\rightarrow\infty}\lim_{m\rightarrow\infty}\left|\int_0^T\int_t^T\big(f(u^{(n)}(\tau))-f(u^{(m)}(\tau)),u^{(n)}_{t}(\tau)-u^{(m)}_{t}(\tau)\big)d\tau dt\right|\\=&\lim_{n\rightarrow\infty}\lim_{m\rightarrow\infty}\left|\int_0^T\big(f(u^{(n)}(t))-f(u^{(m)}(t)),u^{(n)}_{t}(t)-u^{(m)}_{t}(t)\big) dt\right|=0.
\end{split}\end{equation}
By Theorem~$\ref{20-7-27-80}$, we deduce from~$(\ref{21-9-2-66})$,~$(\ref{1-13-12})$ and the precompactness of ~$\varrho^1_{T}(y_1,y_2)=\int_0^T\|\Psi(z_t(t))\|dt,~\rho^2_{T}(y_1,y_2)=\sup_{t\in[0,T]}\|z(t)\|$ that
\begin{equation}\label{20-7-28-1}
\begin{split}
\alpha(S(t)\mathcal{B}_0)\leq &2\bigg\{(\alpha(\mathcal{B}_0))^{-p}+\frac{p}{2}\big(t-(N_0+1)T\big)\Big(T^{\frac{2}{p+2}}+3\cdot(kC_p)^{-\frac{2}{p+2}}\cdot2^{\frac{p}{p+2}}\Big)^{-\frac{p+2}{2}}\bigg\}^{-\frac{1}{p}}, \ \forall t\geq (N_0+1)T.
\end{split}
\end{equation}
By  the arbitrariness of the positive constant~$T$, we have
\begin{equation*}
\begin{split}
\alpha(S(t)\mathcal{B}_0)\leq 2\Big\{(\alpha(\mathcal{B}_0))^{-p}+pkC_p6^{-\frac{p+2}{2}}t\Big\}^{-\frac{1}{p}},\ \forall t> 0.
\end{split}
\end{equation*}
Consequently, problem~$(\ref{wave equa})$-$(\ref{condition})$ possesses a generalized polynomial attractor~$\mathcal{A}^*$ which satisfies~$(\ref{21-4-24-1})$.
\end{proof}
 \section*{Acknowledgement }

The work is supported by National Natural Science Foundation of China (No.11731005; No.11801071).



%% References


\end{document}